\documentclass[11pt]{amsart}
\usepackage{graphic x}
\usepackage{amssymb,amsmath}
\usepackage[margin=1in]{geometry}
\usepackage{setspace}
\usepackage{epstopdf}
\usepackage{amsmath,amsthm,amscd,amssymb}
\usepackage{latexsym}
\usepackage[colorlinks,citecolor=red,pagebackref,hypertexnames=false]{hyperref}
\usepackage{geometry}

\numberwithin{equation}{section}
\theoremstyle{plain}
\newtheorem{theorem}{\bf Theorem}
\newtheorem{lemma}[theorem]{\bf Lemma}
\newtheorem{conjecture}[theorem]{Conjecture}

\newtheorem{proposition}[theorem]{Proposition}
\theoremstyle{definition}

\newtheorem{note}[theorem]{Note}

\newtheorem{case[theorem]}{Case}

\theoremstyle{remark}
\newtheorem{remark}[theorem]{Remark}
\numberwithin{equation}{section}

\begin{document}

\title{Pairs of dot products in finite fields and rings}
\author{David Covert and Steven Senger}

\subitem \email{covertdj@umsl.edu, stevensenger@missouristate.edu}

\thanks{}

\begin{abstract} We obtain bounds on the number of triples that determine a given pair of dot products arising in a vector space over a finite field or a module over the set of integers modulo a power of a prime.  More precisely, given $E\subset \mathbb F_q^d$ or $\mathbb Z_q^d$, we provide bounds on the size of the set
\[\left\{(u,v,w)\in E \times E \times E : u\cdot v = \alpha, u \cdot w = \beta   \right\}\]
for units $\alpha$ and $\beta$.
\end{abstract}

\maketitle

\section{Introduction}

For a subset of a ring, $A \subset R$, the sumset and productset of $A$ are defined as $A +A =\{a + a' : a, a' \in A\}$ and $A \cdot A = \{a \cdot a' : a, a' \in A\}$, respectively.  The sum-product conjecture asserts that when $A \subset \mathbb{Z}$, then either $A+A$ or $A \cdot A$ is of large cardinality.  For example if we take $A \subset \mathbb{Z}$ to be a finite arithmetic progression of length $n$, you achieve $|A+A| = 2n - 1$, whereas $|A \cdot A| \geq c n^2/((\log n)^{\delta} \cdot (\log \log n)^{3/2})$ for some constant $c > 0$ and $\delta = 0.08607\dots$ (\cite{Ford}). When $A \subset \mathbb{Z}$ is a geometric progression of length $n$, we have $|A\cdot A| = 2n - 1$, and yet it is easy to show that $|A + A| = {n+1 \choose 2}$.  For subsets of integers the following conjecture was made in \cite{ES83}.
\begin{conjecture}
Let $A \subset \mathbb{Z}$ with $|A| = n$.  For every $\epsilon > 0$, there exists a constant $C_{\epsilon}>0$ so that
\[
\max(|A+A|,|A\cdot A|) \geq C_{\epsilon} n^{2-\epsilon}.
\]
\end{conjecture}
Much progress been made on the sum-product problem.  The best result to date belongs to Konyagin and Shkredov (\cite{KonShk}), wherein they demonstrated that for a sufficiently large constant $C$, we have the bound
\[
\max(|A+A|,|A\cdot A|) \geq Cn^{4/3 + c}
\]
for any $c < \frac{1}{20598}$, whenever $A$ is a set of real numbers with cardinality $n$.  Work has also been done on analogues of the sum-product problem for general rings (\cite{TaoRing}).  For example, the authors in \cite{HIKR} showed that if $E \subset \mathbb{F}_q^d$ is of sufficiently large cardinality, then we have 
\[
|\{(x,y) \in E \times E : x\cdot y = \alpha\}| = \frac{|E|^2}{q}(1 + \underline{o}(1)),
\]
for any $\alpha \in \mathbb{F}_q^*$.  Here, $\mathbb{F}_q$ is the finite field with $q$ elements, $\mathbb{F}_q^d$ is the $d$-dimensional vector space over $\mathbb{F}_q$, and $\mathbb{F}_q^* = \mathbb{F}_q \setminus\{0\}$.  As a corollary they showed that $|dA^2| := |A \cdot A + \dots + A \cdot A| \supset \mathbb{F}_q^*$, whenever $A \subset \mathbb{F}_q$ is such that $|A| \geq q^{\frac{1}{2} + \frac{1}{2d}}$.  Much work has also been done to give such results when $E$ has relatively small cardinality.  See, for example, \cite{KS} and the references contained therein.

In \cite{BS}, the second listed author and Daniel Barker studied pairs of dot products determined by sets $P \subset \mathbb{R}^2$. In addition to the applications toward the sum-product problem above, the problem of pairs of dot products has applications in coding theory, graph theory, and frame theory, among others (\cite{AB,Bahls,Fickus}).  The main results from \cite{BS} are as follows.
\begin{theorem}\label{steve}
Suppose that $P \subset \mathbb{R}^2$ is a finite point set with cardinality $|P| = n$.  Then, the set
\[
\Pi_{\alpha, \beta}(P) : = \{(x,y,z) \in P \times P \times P : x \cdot y = \alpha, x \cdot z = \beta\}
\]
satisfies the upper bound $|\Pi_{\alpha, \beta}(P)| \lesssim n^2$ whenever $\alpha$ and $\beta$ are fixed, nonzero real numbers.
\end{theorem}
\begin{note}
Here and throughout, we use the notation $X \lesssim Y$ to mean that $X \leq cY$ for some constant $c > 0$. Similarly, we use $X \gtrsim Y$ for $Y \lesssim X$, and we use $X \approx Y$ if both $X \lesssim Y$ and $X \gtrsim Y$. Finally, we write $X \gtrapprox Y$ if for all $\epsilon > 0$, there exists a constant $C_{\epsilon} > 0$ such that $X \gtrapprox C_{\epsilon} q^{\epsilon} Y$.
\end{note}

Theorem \ref{steve} is sharp, as shown in an explicit construction (\cite{BS}). Additionally, they showed the following.

\begin{theorem}\label{sep}
Suppose that $P \subset [0,1]^2$ is a set of $n$ points that obeys the separation condition
\[
\min(|p-q| : p, q \in P, p \neq q) \geq \epsilon.
\]
Then, for $\epsilon > 0$ and fixed $\alpha, \beta \neq 0$, we have
\[
|\Pi_{\alpha, \beta}(P)| \lesssim n^{4/3}\epsilon^{-1}\log\left(\epsilon^{-1}\right).
\]
\end{theorem}
The purpose of this article is to study finite field and finite ring analogues of the results from \cite{BS}.  Our main results are as follows.
\begin{theorem} \label{general}
Given a set, $E \subset \mathbb F_q^2$ or $\mathbb Z_q^d, |E| = n$, and fixed units $\alpha , \beta$, we have the bound
$$|\Pi_{\alpha,\beta}(E)|\lesssim n^2.$$
\end{theorem}

In general, for a set of $n$ points, $E\subset \mathbb F_q^2$, one cannot expect to get an upper bound better than Theorem \ref{general}, as we will show via an explicit construction in Proposition \ref{sharp}. This proof and construction are similar to their analogues in \cite{BS}.  However, if we view the separation condition from Theorem \ref{sep} as it relates to density (as is often the case for translating such results, such as in \cite{IS}), the previous proof techniques yield very little. It turns out that a discrepancy theoretic approach gives more information, as our second main result is for general subsets of $\mathbb{F}_q^d$, for $d\geq 2$, as opposed to just $d=2$.

\begin{theorem}\label{large}
Let $d\geq 2$, $E \subset \mathbb{F}_q^d$, and suppose that $\alpha, \beta \in \mathbb{F}_q.$  Then we have the bound
\[
|\Pi_{\alpha, \beta}(E)| = \frac{|E|^3}{q^2}(1 + \underline{o}(1)),
\]
for $|E| \gtrapprox q^{\frac{d+1}{2}}$ when $\alpha, \beta \in \mathbb{F}_q^*$, and for $|E| \gtrapprox q^{\frac{d+2}{2}}$ otherwise.
\end{theorem}
Note that Theorem \ref{large} gives a quantitative version of Theorem \ref{general} at least for sets $E \subset \mathbb{F}_q^2$ in the range $|E| \gtrapprox q^{3/2}$.  

The proof of Theorem \ref{large} relies on adapting the exponential sums found in the study of single dot products (\cite{HIKR}).  Since the results from \cite{HIKR} were extended to general rings $\mathbb{Z}_q^d$ in \cite{CIP}, Theorem \ref{large} also easily extends to rings.  Here and throughout, $\mathbb{Z}_q$ denotes the set of integers modulo $q$, $\mathbb{Z}_q^{\times}$ is the set of units in $\mathbb{Z}_q$, and $\mathbb{Z}_q^d = \mathbb{Z}_q \times \dots \times \mathbb{Z}_q$ is the $d$-rank free module over $\mathbb{Z}_q$.  For $E \subset \mathbb{Z}_q^d$, we define $\Pi_{\alpha, \beta}(E)$ exactly as before.
\begin{theorem}\label{LargeZqd}
Suppose that $E \subset \mathbb{Z}_q^d$, where $q = p^{\ell}$ is the power of a prime $p \geq 3$.  Then for units $\alpha, \beta \in \mathbb{Z}_q^{\times}$, we have
\[
|\Pi_{\alpha, \beta}(E)| = \frac{|E|^3}{q^2}(1 + \underline{o}(1))
\]
whenever $|E| \gtrapprox q^{\frac{d(2 \ell - 1)}{2 \ell} + \frac{1}{2 \ell}}$.  In particular, 
\[ 
|\Pi_{\alpha, \beta}(E)| \lesssim |E|^2
\] for sets $E \subset \mathbb{Z}_q^2$ of sufficiently large cardinality.
\end{theorem}

\begin{remark}
Notice that the proofs of Theorems \ref{large} and \ref{LargeZqd} provide both a lower and upper bound on the cardinality of $\Pi_{\alpha, \beta}(E)$, though we could achieve the upper bound $|\Pi_{\alpha, \beta}(E)| \lesssim q^{-2} |E|^3$ if we relaxed the condition $|E| \gtrapprox q^{\frac{d+1}{2}}$ to simply $|E| \gtrsim q^{\frac{d+1}{2}}$, for example.
\end{remark}

\section{Explicit constructions}
\subsection{Sharpness of Theorem \ref{general}}
We construct explicit sharpness examples for $\mathbb F_q^2$. The same constructions can be modified to yield sharpness in $\mathbb Z_q^2$ as well.
\begin{proposition}\label{sharp}
Given a natural number $n\lesssim q$, and elements $\alpha,\beta\in \mathbb F_q^*$, there is a set, $E \subset \mathbb F_q^2$ for which $|E|=n$ and
$$|\Pi_{\alpha,\beta}(E)|\approx n^2.$$
\end{proposition}
\begin{proof}
Let $u$ be the point with coordinates $(1,1)$. Now, distribute up to $\left\lfloor \frac{n-1}{2} \right\rfloor$ points along the line $y=\alpha -x$, and distribute the remaining up to $\left\lceil \frac{n-1}{2} \right\rceil$ points along the line $y=\beta -x$. If there are any points left over, put them anywhere not yet occupied.\footnote{This is just in the case that $(1,1)$ is on one of the lines or $\alpha =\beta$.} Clearly, there are at least $|E|^2$ pairs of points $(b,c)$, where $q$ is chosen from the first line, and $r$ is chosen from the second. Notice that $u$ contributes a triple to $\Pi_{\alpha,\beta}(E)$ for each such pair, giving us
$$|\Pi_{\alpha,\beta}(E)|\approx n^2.$$
\end{proof}

\subsection{The special case $\alpha=\beta=0, d = 2$}
\begin{proposition}\label{zeroDPs}
There exists a set $E \subset \mathbb{F}_q^2$ of cardinality $|E| = n<2q$ for which
$$|\Pi_{0,0}(P)|\approx n^3.$$
\end{proposition}
\begin{proof}
Select $\left\lceil\frac{n}{2}\right\rceil$ points with zero $x$-coordinate, and $\left\lceil\frac{n}{2}\right\rceil$ points with zero $y$-coordinate. Now, for each of the points with zero $x$-coordinate, there are about $\left(\frac{n}{2}\right)\left(\frac{n}{2}\right)$ pairs of points with zero $y$-coordinate. Notice that any point chosen with zero $x$-coordinate will have dot product zero with each point from the pair chosen with zero $y$-coordinate. Therefore, each of these $\frac{1}{8}n^3$ triples will contribute to $\Pi_{0,0}(E)$.

We can get just as many triples that contribute to $\Pi_{0,0}(E)$ by taking single points with zero $y$-coordinate, and pairs of points with zero $x$-coordinate. In total, we get
$$|\Pi_{0,0}(P)|\approx\frac{1}{8}n^3+\frac{1}{8}n^3 \approx n^3.$$
\end{proof}

\section{Proofs of main results}\label{proofs}

\subsection{Proof of Theorem \ref{general}}\label{genProof}
This proof is a modified version of the proof of Theorem 1 in \cite{BS}, to which we refer the reader for a more detailed exposition.
\begin{proof}
We will simultaneously prove this for $E\subset \mathbb F_q^2$ and $E \subset \mathbb Z_q^2.$ Here, we will use $R_q$ to denote either $\mathbb F_q$ or $\mathbb Z_q$, and we will be more specific when necessary.

Our basic idea is to consider pairs of points $(v,w) \in E \times E$ and obtain a bound on the number of possible candidates for $u$ to contribute a triple of the form $(u,v,w)$ to $\Pi_{\alpha,\beta}(E)$. Consider $a  = (a_x, a_y) \in R_q^2$, and notice that for a point $v\in E$, the set of points $L_\alpha(v)$ that determine the dot product $\alpha$ with $v$  forms a line.
\begin{equation}\label{lineEq}
L_\alpha(v)=\left\{(x,y)\in R_q^2: xv_x+yv_y=\alpha\right\}.
\end{equation}
Also, $v$ lies on a unique line containing the origin. We similarly define $L_\beta(v).$  Now, consider a second point $w\in E$. It is easy to see that if $|L_\alpha(v)\cap L_\beta(w)|> 1$, then $v$ and $w$ lie on the same line through the origin which implies that if $v$ and $w$ are on different lines through the origin, then $|L_\alpha(v)\cap L_\beta(w)|\leq 1$. We will use this dichotomy to decompose $E\times E$ into two sets:
\begin{align*}
A&=\{(v,w)\in E \times E : |L_\alpha(v)\cap L_\beta(w)|\leq 1, |L_\alpha(w)\cap L_\beta(v)|\leq 1\}
\\
B&=(E \times E)\setminus A.
\end{align*}

Given $(v,w) \in A$, the pair can only be the last coordinates of at least triple in $\Pi(E)$. This is of course only if $L_\alpha(v)\cap L_\beta(w)$ is a point in $E$. As there are no more than $|E|^2$ choices for pairs $(v,w)\in A,$ the contribution to $\Pi(E)$ by point pairs in $A$ is at most $|E|^2$.

The analysis on the set of pairs in $B$ is a bit more delicate. Consider an arbitrary pair, $(v,w)\in B$. Without loss of generality (possibly exchanging $v$ with $w$ or $\alpha$ with $\beta$) suppose $|L_\alpha(v)\cap L_\beta(w)|>1.$ Then the following holds:
\begin{align*} |L_\alpha(v)\cap L_\beta(w)| &>1
\\
\left|\left\{(x,y)\in R_q^2: xv_x+yv_y=\alpha\right\}\cap \left\{(x,y)\in R_q^2: xw_x+yw_y=\beta\right\}\right|&>1
\\
\left|\left\{(x,y)\in R_q^2: xv_x+yv_y=\alpha \text{ and } xw_x+yw_y=\beta\right\}\right|&>1.
\end{align*}
Namely, there will be more than one point with coordinates $(x,y)\in R_q^2$ satisfying
\begin{equation}\label{lambda}
xv_x+yv_y=\alpha\left(\frac{xw_x+yw_y}{\beta}\right)=\frac{\alpha}{\beta}(xw_x+yw_y).
\end{equation}
Note that $\beta$ is a unit, and hence the quantity $\alpha/\beta$ is well defined. This restriction tells us that if $|L_\alpha(v)\cap L_\beta(w)|>1$, then $|L_\alpha(v)\cap L_\beta(w')|=0$, for any $w'\neq w$. This should not be surprising for if $\alpha =\beta$, then $L_\alpha(v)=L_\beta(w)$ forces $v=w$.

We pause for a moment to introduce an equivalence relation, say $\sim$, on the set of lines. Two lines $L_\alpha(v)$ and $L_\beta(w)$ are equivalent under $\sim$ if one can be translated to become a (possibly improper) subset of the other. It is clear that if $|L_\alpha(v)\cap L_\beta(w)|>1$ then $L_\alpha(v) \sim L_\beta(w)$. The equivalence classes of $\sim$ keep track of the different ``directions" that lines can have. So we can easily see that $L_\alpha(v)\sim L_\beta(v)$. Take note that if $R_q=\mathbb Z_q$ it is possible for two distinct lines to intersect in more than one point. 

If $|L_\alpha(v)\cap L_\beta(w)|>1$, then the pair $(v,w)$ has no more than $\min\{|L_\alpha(v)|,|L_\beta(w)|\}$ possible choices for $u$ to contribute a triple of the form $(u,v,w)$ to $\Pi_{\alpha,\beta}(E)$. Now, we see that any other pair of points, say $(v',w')$, with $|L_\alpha(v')\cap L_\beta(w')|>1$ and with $L_\alpha(v) \sim L_\alpha(v')$, will have $L_\alpha(v)\cap L_\alpha(v')=\emptyset,$ and $L_\beta(w) \sim L_\beta(w')$, will have $L_\beta(w)\cap L_\beta(w')=\emptyset.$ So any point $u$ that contributes to a triple of the form $(u,v,w)\in\Pi_{\alpha,\beta}(E)$, can only contribute to a triple with a single pair $(v,w)$ when $L_\alpha(v) \sim L_\beta(w)$.

Therefore, given any single equivalence class of $\sim$, there can be no more than $|E|$ choices for $u$ to contribute a triple of the form $(u,v,w)$ to $\Pi_{\alpha,\beta}(E)$ with $(v,w)\in B.$ As there are no more than $|E|$ possible choices for equivalence classes of $L_\alpha(v)$ (as each point has only one associated equivalence class of $\sim$), there are no more than $|E|^2$ triples of the form $(u,v,w) \in \Pi_{\alpha,\beta}(E)$ with $(v,w)\in B.$

\end{proof}

\subsection{Proof of Theorem \ref{large}}
\begin{proof} Let $\chi$ denote the canonical additive character of $\mathbb{F}_q$. By orthogonality, we have
\begin{align*}
|\Pi_{\alpha, \beta}(E)| &=  |\{(x,y,z)\in E\times E\times E: x\cdot y = \alpha, x \cdot z = \beta\}\\
&=q^{-2}\sum_{s,t \in \mathbb F_q}\sum_{x,y,z \in E} \chi(s(x\cdot y-\alpha))\chi(t(\beta-x \cdot z))\\
&=q^{-2}\sum_{s,t \in \mathbb{F}_q} \sum_{x, y, z \in E} \chi(s\alpha) \chi(-t\beta) \chi(x \cdot (sy - t z))
\\
&:= I + II + III,
\end{align*}
Where $I$ is the term with $s=t=0$, $II$ is the term with exactly one of $s$ or $t$ equal to zero, and $III$ is the term with $s$ and $t$ both nonzero. Clearly 
\[
I = q^{-2} \sum_{s = t = 0} \sum_{x,y , z\in E} \chi(s\alpha) \chi(-t\beta) \chi(x \cdot (sy - t z)) = |E|^3q^{-2}.
\]  For the second and third sums, we need the following known results.
\begin{lemma}[\cite{HIKR}] For any set $E \subset \mathbb{F}_q^d$, we have the bound
\begin{equation}\label{ell1}
\sum_{s \neq 0} \sum_{x,y \in E} \chi(s(x \cdot y - \gamma)) \leq  |E|q^{\frac{d+1}{2}}\lambda(\gamma),
\end{equation}
where $\lambda(\gamma) = 1$ for $\gamma \in \mathbb{F}_q^*$ and $\lambda(0) = \sqrt{q}$.  Furthermore, we have
\begin{equation}\label{ell2}
\sum_{s, s' \neq 0} \sum_{\substack{y, y' \in E \\ sy = s'y'}}  \chi(\alpha(s' - s)) \leq |E|q\lambda(\gamma).
\end{equation}
\end{lemma}
Note that the quantities in the above Lemma can be shown to be real numbers, so there is no need for absolute values. Now, separating the $II$ term into two sums, each with exactly one of $s$ or $t$ zero,
\begin{align*}
II &= q^{-2}|E|\left( \sum_{s \neq 0} \sum_{x, y \in E} \chi(s (x \cdot y - \alpha)) + \sum_{t \neq 0} \sum_{x, z \in E} \chi(t (x \cdot z - \beta))\right)
\end{align*}
From \eqref{ell1}, it follows that $|II| \leq |E|^2q^{\frac{d-3}{2}}(\lambda(\alpha) +\lambda(\beta))$.  Finally, by the triangle-inequality, dominating a nonnegative sum over $x\in E$ by the same nonnegative sum over $x\in \mathbb F_q^d$, and applying Cauchy-Schwarz we have
\begin{align*}
|III| &\leq q^{-2} \sum_{x \in E} \left|\sum_{s \neq 0} \sum_{y \in E}\chi(s(x \cdot y - \alpha) )\right|\left| \sum_{t \neq 0}\sum_{z \in E}  \chi(t(x \cdot z - \beta))\right|
\\
&\leq q^{-2} \sum_{x \in \mathbb F_q^d} \left|\sum_{s \neq 0} \sum_{y \in E}\chi(s(x \cdot y - \alpha) )\right|\left| \sum_{t \neq 0}\sum_{z \in E}  \chi(t(x \cdot z - \beta))\right|
\\
&\leq q^{-2} \left( \sum_{x \in \mathbb{F}_q^d} \left| \sum_{s \neq 0} \sum_{y \in E} \chi(s(x \cdot y - \alpha)) \right|^2 \right)^{1/2}\left( \sum_{x \in \mathbb{F}_q^d} \left| \sum_{t \neq 0} \sum_{z \in E}\chi(t(x \cdot z - \beta)) \right|^2 \right)^{1/2}
\\
&:= q^{-2} III_{\alpha} \cdot III_{\beta}
\end{align*}
Now,
\begin{align*}
III_{\alpha}^2 &=\sum_{x \in \mathbb{F}_q^d} \left| \sum_{s \neq 0} \sum_{y \in E} \chi(s (x \cdot y - \alpha)) \right|^2 
\\
&= \sum_{x} \sum_{s, s' \neq 0} \sum_{y, y' \in E} \chi(s (x \cdot y - \alpha)) \chi(-s' (x \cdot y' - \alpha)) 
\\
&= \sum_{x} \sum_{s, s' \neq 0} \sum_{y, y' \in E} \chi(\alpha(s' - s)) \chi(x \cdot (sy - s' y'))
\\
&= q^{d} \sum_{s, s' \neq 0} \sum_{\substack{y, y' \in E \\ sy = s'y'}} \chi(\alpha(s' - s))
\\
&\leq q^{d+1}|E| \lambda(\alpha)^2
\end{align*}
by \eqref{ell2}.  Similarly, we have $III_{\beta} \leq \sqrt{ q^{d+1}|E|} \lambda(\beta)$.  Combining these estimates yields
\[
|III| \leq q^{d-1}|E| \lambda(\alpha)\lambda(\beta).
\]
This completes the proof as we have
\[
|\Pi_{\alpha, \beta}(E)| = \frac{|E|^3}{q^2} + R_{\alpha, \beta},
\]
where
\[
|R_{\alpha,\beta}|  \leq |E|^2q^{\frac{d-3}{2}}(\lambda(\alpha) +\lambda(\beta)) + q^{d-1}|E| \lambda(\alpha)\lambda(\beta).
\]

\end{proof}

\subsection{Proof of Theorem \ref{LargeZqd}}
The proof will imitate that of Theorem \ref{large}, so we omit some of the details.  Let $\chi(\sigma) = \exp(2 \pi i \sigma / q)$ be the canonical additive character of $\mathbb{Z}_q$, and identify $E$ with its characteristic function.  We use the following known bounds for dot-product sets in $\mathbb{Z}_q^d$.
\begin{lemma}[\cite{CIP}]\label{LemmaZqd}
Suppose that $E \subset \mathbb{Z}_q^d$,  where $q = p^{\ell}$ is the power of an odd prime.  Suppose that $\gamma \in \mathbb{F}_q^{\times}$ is a unit.  Then we have the following upper bounds.
\begin{equation}\label{L1Zqd}
\sum_{j \in \mathbb{Z}_q \setminus\{0\}} \sum_{x, y \in E} \chi(j (x \cdot y - \gamma )) \leq 2 |E| q^{\left(\frac{d-1}{2}\right) \left(2 -  \frac{1}{\ell}\right)+1}
\end{equation}
and
\begin{equation} \label{L2Zqd}
\sum_{s,s' \in \mathbb{Z}_q \setminus\{0\}} \sum_{\substack{y, y' \in E \\ sy = s'y'}} \chi(\gamma(s' - s)) \leq 2|E| q^{\frac{\ell d  - d + 1}{\ell}}
\end{equation}
\end{lemma}
\begin{note}
The authors in \cite{CIP} actually gave a slightly different bound than those in Lemma \ref{LemmaZqd}.  For example in \eqref{L1Zqd}, they showed
\[
\sum_{j \in \mathbb{Z}_q \setminus\{0\}} \sum_{x, y \in E} \chi(j (x \cdot y - \gamma )) \leq \sum_{i = 0}^{\ell - 1} |E| q^{\left(\frac{d-1}{2}\right)\left( 1 + \frac{i}{\ell}\right)} \leq \ell |E| q^{\left(\frac{d-1}{2}\right) \left(2 -  \frac{1}{\ell}\right)+1}.
\]
However, summing the geometric series removes the factor of $\ell$ in the estimate.  Likewise, a factor of $\ell$ can be removed from the estimate in \eqref{L2Zqd}.
\end{note}

We proceed as before.  Write
\begin{align*}
|\Pi_{\alpha, \beta}| = \frac{|E|^3}{q^2} + II + III,
\end{align*}
where
\[
II := |E|q^{-2} \left( \sum_{s \neq 0} \sum_{x,y \in E} \chi(s \cdot (x \cdot y - \alpha)) +  \sum_{t \neq 0} \sum_{x,z \in E} \chi(s \cdot (x \cdot z - \beta))\right)
\]
and
\begin{align*}
III := q^{-2} \sum_{x \in E} \left(\sum_{s \neq 0} \sum_{y \in E} \chi(-s\alpha) \chi(s(x \cdot y))\right)\left(\sum_{t \neq 0} \sum_{z \in E} \chi(-t \beta) \chi(t(x \cdot z))\right).
\end{align*}
Applying Lemma \ref{LemmaZqd}, we see that
\begin{align*}
|II| &\leq 4 |E|^2 q^{-2} q^{\frac{d(2 \ell - 1) + 1}{2 \ell}},
\end{align*}
while
\begin{align*}
|III| &\leq q^{-2} \left( \sum_{x \in \mathbb{F}_q^d}  \left| \sum_{s \neq 0} \sum_{y \in E} \chi(-s\alpha) \chi(s(x \cdot y))\right|^2 \right)^{1/2}\left(\sum_{x \in \mathbb{F}_q^d} \left| \sum_{t \neq 0} \sum_{z \in E} \chi(-t \beta) \chi(t(x \cdot z))\right|^2\right)^{1/2}
\\
&\leq 2|E|q^{-2} q^{\frac{d(2 \ell - 1)}{ \ell} + \frac{1}{\ell}}.
\end{align*}
This completes the proof.


\begin{thebibliography}{99}


\bibitem{AB} J. Alvarez-Bermejo, J. A. Lopez-Ramos, J. Rosenthal, D. Schipani, {\it Managing key multicasting through orthogonal systems}, arXiv:1107.0586v2, (2015).

\bibitem{Bahls} P. Bahls, {\it Channel assignment on Cayley graphs.} J. Graph Theory, 67: 169--177, (2011). doi: 10.1002/jgt.20523

\bibitem{BS} D. Barker and S. Senger, {\it Upper bounds on pairs of dot products}, arXiv:1502.01729 (2015).

\bibitem{Fickus} J. J. Benedetto and M. Fickus, {\it Finite normalized tight frames,} Adv. Comput. Math. 18, pp. 357--385 (2003).

\bibitem{CIP} D. Covert, A. Iosevich, J. Pakianathan, \emph{Geometric configurations in the ring of integers modulo $p^{\ell}$}, Indiana University Mathematics Journal, 61 (2012), 1949-1969.

\bibitem{ES83} P. Erd\H os and E. Szemer\'edi, {\it On sums and products of
integers}. Studies in pure mathematics, 213--218, Birkh\"auser, Basel, 1983.

\bibitem{Ford} K. Ford, {\it Integers with a divisor in an interval}, Annals of Math. \textbf{168} (2), 367-433, 2008.



\bibitem{HIKR} D. Hart, A. Iosevich D. Koh, M. Rudnev, {\it Averages over hyperplanes, sum-product theory in vector spaces over finite fields and the Erd\H os-Falconer distance conjecture}, Trans. Amer. Math. Soc. 363 (2011), no. 6, 3255--3275.




\bibitem{IS} A. Iosevich and S. Senger, {\it Orthogonal systems in vector spaces over finite fields,} Electronic J. of Combinatorics, Volume 15, December (2008).


\bibitem{KS} N. H. Katz, C. Y. Shen, {\it A slight improvement to Garaev's sum product estimate.}  Proc. Amer. Math. Soc. \textbf{136} (2008), no. 137, 2499-2504.

\bibitem{KonShk} S. Konyagin, I. Shkredov, {\it On sum sets of sets, having small product set}, arXiv:1503.05771 (2015).





\bibitem{TaoRing} T. Tao, {\it The sum-product phenomenon in arbitrary rings}. Contrib. Discrete Math. 4 (2009), no. 2, 59--82.



\end{thebibliography}
\end{document}